\documentclass[11pt,reqno]{amsart}
\usepackage{amssymb,amsthm,amsmath,epsfig,latexsym}
\usepackage{graphicx}
\usepackage{calc,times}
\usepackage{hyperref}
\usepackage{geometry} 
\newcommand{\rank}{\mathop{\rm rank}\nolimits}
\newcommand{\supp}{\mathop{\rm Supp}\nolimits}

\newcommand{\m}{{\mathfrak m}}
\newcommand{\fitt}{{\rm Fitt}}

\newcommand{\C}{{\mathcal{C}}}

\newcommand{\K}{\mathbb K}
\newcommand{\spn}{{\rm Span }}  
\newcommand{\M}{\mathsf M}

\newcommand{\hi}{{\rm ht}}

\newtheorem{defn0}{Definition}[section]
\newtheorem{prop0}[defn0]{Proposition}
\newtheorem{quest0}[defn0]{Question}
\newtheorem{thm0}[defn0]{Theorem}
\newtheorem{lem0}[defn0]{Lemma}
\newtheorem{corollary0}[defn0]{Corollary}
\newtheorem{example0}[defn0]{Example}
\newtheorem{remark0}[defn0]{Remark}
\newtheorem{prob0}[defn0]{Problem}
\newtheorem{conj}{Conjecture}

\newenvironment{prop}{\begin{prop0}}{\end{prop0}}

\newenvironment{thm}{\begin{thm0}}{\end{thm0}}
\newenvironment{lem}{\begin{lem0}}{\end{lem0}}
\newenvironment{cor}{\begin{corollary0}}{\end{corollary0}}
\newenvironment{ex}{\begin{example0}\rm}{\end{example0}}
\newenvironment{rem}{\begin{remark0}\rm}{\end{remark0}}

\numberwithin{equation}{subsection}  

\begin{document}

\title{Generalized star configurations and the Tutte polynomial}
\author{Benjamin Anzis, Mehdi Garrousian and \c{S}tefan O. Toh\v{a}neanu}

\subjclass[2010]{Primary 05B35; Secondary: 13D40, 94B27, 11T71, 14G50} \keywords{star configuration, Tutte polynomial, Hilbert polynomial, generalized Hamming weights, free resolution. \\ \indent Anzis' address: Department of Mathematics, University of Idaho, Moscow, Idaho 83844-1103, USA, Email: anzi4123@vandals.uidaho.edu\\ \indent Garrousian's Address: Departamento de Matem\'aticas, Universidad de los Andes, Cra 1 No. 18A-12, Bogot\'a, Colombia, Email: m.garrousian@uniandes.edu.co\\ \indent Tohaneanu's Address: Department of Mathematics, University of Idaho, Moscow, Idaho 83844-1103, USA, Email: tohaneanu@uidaho.edu, Phone: 208-885-6234, Fax: 208-885-5843.}

\begin{abstract}
From the generating matrix of a linear code one can construct a sequence of generalized star configurations which are strongly connected to the generalized Hamming weights and the underlying matroid of the code. When the code is MDS, the matrix is generic and we obtain the usual star configurations. In our main result, we show that the degree of a generalized star configuration as a projective scheme is determined by the Tutte polynomial of the code. In the process, we obtain preliminary results on the primary decomposition of the defining ideals of these schemes. Additionally, we conjecture that these ideals have linear minimal free resolutions and prove partial results in this direction.
\end{abstract}

\maketitle

\section{Introduction}

Let $\mathbb K$ be any field, and let $\mathcal C=(\ell_1,\ldots,\ell_n)$
denote a collection of $n$ linear forms in $R:=\mathbb K[x_1,\ldots,x_k]$,
possibly with repetitions. Throughout these notes, $\langle
\ell_1,\ldots,\ell_n\rangle=\langle x_1,\ldots,x_k\rangle=:\mathfrak m$. Let
$a\in\{1,\ldots,n\}$ and let $I_a(\C)\subset R$ be the ideal of $R$ generated
by all $a$-fold products of the linear forms in $\mathcal C$, i.e.
  \[
    I_a(\C)=\langle \ell_{i_1}\cdots\ell_{i_a}|1\leq i_1<\cdots<i_a\leq n \rangle.
  \]
The projective scheme with defining ideal $I_a(\C)$ will be called a {\em generalized star configuration scheme}.

If every $k$-subset of the linear forms in $\mathcal C$ is linearly independent, then $V(I_a(\C))$ will be a star configuration of codimension $n-a+1$ in the usual sense, as in \cite{GHM} and the citations therein, where their homological properties are studied. In particular, it is known that they are arithmetically Cohen-Macaulay, the $h$-vectors and the degrees are known (see
\cite[Proposition 2.9]{GHM}), and their coordinate rings have Eagon-Northcott
free resolutions, which are in fact linear (see\cite[Remark 2.11]{GHM}). When it comes to generalized star configurations different than the usual ones, almost nothing is known about the same properties listed above. Generalized star configurations (to be specific, their defining ideals) first occurred relatively recently in the coding theoretical context, and therefore they did not receive enough attention from the commutative algebraists. With certainty we can say that these schemes cannot be fully understood unless one studies them from both perspectives: coding theory and commutative algebra. For instance, as we shall see below, the dimensions of these schemes are entangled with the generalized Hamming weights of the associated linear code.

The main goal of these notes is to describe the degree of any generalized star
configuration scheme from the coefficients of the Tutte polynomial of the
matroid of the associated linear code (see Theorem \ref{main}). In obtaining
this result, we give a partial description of the the primary decomposition of the
ideals $I_a(\C)$. Conjecturally, we believe that for any $a$ and $\C$ the
ideal $I_a(\C)$ has a linear graded free resolution.

We recommend \cite{JuPe} for background information on linear codes, basic
matroid theory, and the Tutte polynomial; we recommend \cite{Ei0} as a
reference on the basic commutative/homological algebra used in these notes.

\subsection{Linear codes.} Let us fix an arbitrary field $\mathbb{K}$. An
$[n,k,d]$-linear code $\mathcal C$ over $\K$ is a $k$-dimensional vector space
embedded in $\K^n$ as the image of a generating matrix $G$, where $\rank
(G)=k$. In canonical bases, we write
  \[
    G=\left[\begin{array}{cccc}a_{11}&a_{12}&\cdots&a_{1n}\\ a_{21}&a_{22}&\cdots&a_{2n}\\\vdots&\vdots&
    &\vdots\\ a_{k1}&a_{k2}&\cdots&a_{kn}\end{array}\right],
  \]
where we understand that $a_{ij}\in\mathbb K$, and $\mathcal C$ is the image of the injective linear map
\[\phi:\mathbb K^k\stackrel{G}\longrightarrow \mathbb K^n.
\]
A {\em codeword} is a vector in the row space of $G$, and a subcode is a
subspace, $n$ is the {\em length} of $\mathcal C$, $k$ is the {\em dimension}
of $\mathcal C$, and $d$ is the {\em minimum distance (or Hamming distance)
of} $\mathcal C$, the smallest number of non-zero entries in a non-zero
codeword. These numbers are called the {\em parameters of $mathcal C$}. It is worth mentioning that if $n$ doesn't count any zero columns that possible occur in $G$, then this parameter is called the {\em effective length} of $\mathcal C$. Throughout this paper, most of the time the linear codes considered are derived from collections of (nonzero) linear forms, so their generating matrices $G$ do not have any zero columns.

\begin{rem} We will abuse notation and say that $\mathcal
C=(\ell_1,\ldots,\ell_n)$ is the collection of linear forms in $R:=\mathbb
K[x_1,\ldots,x_k]$ dual to the columns of $G$. When starting with a
collection of linear forms (arrangement of hyperplanes), the associated
linear code is the linear code whose generating matrix's columns are dual to
the given linear forms. If any $k$ of the linear forms are linearly
independent (the setup of star configurations), then the associated linear
code is maximum distance separable (MDS), meaning that $d=n-k+1$, which is
the maximum possible.

The ideals $I_a(\C)$ of a linear code $\C$ first appeared in \cite{DePe},
where the authors show that $d$ is the greatest index $a$ such that
$V(I_a(\C))\subset {\mathbb{P}}^{k-1}$ is empty. In \cite[Theorem 3.1]{To1},
this was refined as follows: for all $a=1,\dots, d$, $I_a(\C)=\m^a$. Moreover,
$V(I_{d+1}(\C))$ is a zero dimensional scheme where each point corresponds to
a minimum distance projective codeword. These ideas are further developed in
\cite{GaTo}, where a Fitting module $\fitt(\C)$ is introduced and it is shown
that the $\alpha$-invariant (the lowest homogeneous degree in which a module
is generated) of $\m\fitt(\C)$ determines the minimum distance.
\end{rem}

Let $\mathcal C$ be an $[n,k,d]$-linear code. Let $\mathcal D\subseteq \mathcal C$ be a subcode. The support of $\mathcal D$ is $$Supp(\mathcal D):=\{i :\exists (x_1,\ldots,x_n)\in \mathcal D \mbox{ with }x_i\neq 0\}.$$ Let $m(\mathcal D):=|Supp(\mathcal D)|$ be the cardinality of the support of $\mathcal D$.

For any $r=1,\ldots,k$, the {\em $r^{\rm {th}}$ generalized Hamming weight of $\mathcal C$} is the positive number $$d_r(\mathcal C):=\min_{\mathcal D\subseteq \mathcal C,\,\dim\mathcal D=r}m(\mathcal D).$$ By convention, $d_0(\mathcal C) = 0$.

Suppose the generating matrix of $\mathcal C$ (which is a $k\times n$ matrix
of rank $k$) is in the reduced echelon form: $$G:=(I_k|A).$$ The parity check
matrix (which is an $(n-k)\times n$ matrix) is $$H:=(-A^T|I_{n-k}),$$ and has
the property that $G \cdot H^T=0$, leading to the fact that $H$ is the
generating matrix for the dual code of $\mathcal C$, denoted $\mathcal
C^{\perp}$. $G$ and $H$ each determine a matroid that we will respectively
denote by $\M=\M(\C)$ with rank function $r$ and $\M^*=\M(\C^{\perp})$ with
rank function $r^*$. This means for instance that if $I\subseteq [n]$, then
$r(I)=\rank G_I$, where $G_I$ is the $k\times |I|$ minor of $G$ determined by
$I$.

Among many other formulations, a matroid can also be described by  its closure operator. Evidently, if $I\subseteq J$, then $\ker G_I\supseteq \ker G_J$. In our context, if $I\subseteq [n]$, then the closure of $I$, denoted $cl(I)$, is the largest superset $J$ with $\ker G_I=\ker G_J$.

A loop in a matroid is an element of rank zero and a coloop is a loop in the dual matroid. More explicitly, $i\in[n]$ is a loop if the $i^{\textrm{th}}$ column of $G$ is zero and it is a coloop if $r([n]\setminus i)<r([n])$.

There are several constructions of minors of a given matroid. Most notably,
given an element $i$ in a matroid $\M$, one can construct the deletion
$\M'=\M\setminus i$ and the contraction $\M''=\M/i$ on the ground set
$[n]\setminus i$, with rank function given by
  \[
    r'(I):=r(I), \quad r''(I):=r(I\cup i)-r(i).
  \]
See \cite{JuPe} for further details.

\medskip

\begin{thm}\cite{We}\label{thm:We}\leavevmode

\begin{enumerate}
  \item $d_1(\mathcal C)$ is the minimum distance $d$ of $\mathcal C$.
  \item $1\leq d_1(\mathcal C)<d_2(\mathcal C)<\cdots<d_k(\mathcal C)\leq n$.
  \item $d_r(\mathcal C)\leq n-k+r$.
  \item $d_r(\C)=\min\{|I|: |I|-r^*(I)=r\}$.
  \item $\{d_r(\mathcal C):1\leq r\leq k\}=\{1,\ldots,n\}\setminus\{n+1-d_s(\mathcal C^{\perp}):1\leq s\leq n-k\}$.
\end{enumerate}
\end{thm}

Note that since we do not allow zero columns in the generating matrix (i.e., loopless matroid), $d_k(\C)$ is always equal to $n$. Also, since the rank of the dual matroid is given by
\[
r^*(I)=r([n]\setminus I)+|I|-r(\M),
\]
 we have the following immediate corollary.

\begin{cor}\label{genhamming} Let $0\leq j\leq k$. The $j^{\rm {th}}$ generalized Hamming weight in terms of $\M$ is characterized as follows.
 \[
  d_j(\C)=\min\{|I|:r([n]\setminus I)=k-j\}=n - \max\{|J|: r(J)=k-j\}
 \]
\end{cor}

By definition, {\em the Tutte polynomial} of a matroid $\M$ on the ground set $[n]$ with rank function $r$ is
\[
T_{\M}(x,y)=\sum_{I\subseteq [n]}(x-1)^{k-r(I)}(y-1)^{|I|-r(I)}.
\]
We use the shorthand notation $T_{\C}(x,y)$ to denote the Tutte polynomial of $\M(\C)$.

By \cite[Formula 11.7]{Du}, the generalized Hamming weights can be read off
the Whitney polynomial, which is $T_{\C}(x+1,y+1)$. To be in line with
\cite{Be} and \cite{GaTo}, and with some of the results that will appear later
on, we use the following version of Duursma's formula.

\begin{lem}\label{lemma:Tutte}
Consider $T_{\mathcal C}(x+1,y)=\sum_{i,j} c_{i,j} x^i y^j$ and let $p_r=\max \{ j: c_{r,j}\neq 0\}$. Then, the $r^{\rm {th}}$ generalized Hamming weight is determined by $d_r(\C)=n-p_r-k+r$.
\end{lem}

\begin{ex}\label{example0}
Consider the $[3,2,2]$-linear code (over any field), given by the generating matrix
 \[
  G=\begin{bmatrix}
             1 & 0 & 1\\
             0 & 1 & 1
    \end{bmatrix}.
 \]
The Tutte polynomial is equal to
\begin{eqnarray}
T_{\C}(x,y)&=&\underbrace{(x-1)^{2-0}(y-1)^{0-0}}_{\emptyset}+3\underbrace{(x-1)^{2-1}(y-1)^{1-1}}_{\{i\}}+ 3\underbrace{(x-1)^{2-2}(y-1)^{2-2}}_{\{i,j\}}\nonumber\\
&&+\underbrace{(x-1)^{2-2}(y-1)^{3-2}}_{\{1,2,3\}}\nonumber\\
&=&x^2+x+y,\nonumber
\end{eqnarray}
and hence $T_{\C}(x+1,y)=x^2+3x+y+2$, giving $p_1=p_2=0$, and therefore $d_1(\C)=3-0-2+1=2$, and $d_2(\C)=3-0-2+2=3$.
\end{ex}

\begin{rem} \label{lattice} The codes live in the row space of $G$ and the
matroid is based on the columns. Let us explain the interplay between these
different points of view. Let $I\subseteq [n]$ have rank $k-r$, for some
$r$, and consider the minor $G_I$. By definition, the rank of this matrix is
equal to the rank of $I$ in the matroid. The kernel of $G_I$ is of dimension
$r=k-(k-r)$. Let $v_1,\dots, v_r$ be a basis for this kernel and consider
$w_i=v_i G$, for $i=1,\dots, r$. The fact that $G$ is full rank implies that
$\{w_1,\dots, w_r\}$ is linearly independent. Let $\C_I=\spn \{w_1,\dots,
w_r\}$. It is immediate that $\C_I=\C_{cl(I)}$, where $cl(I)$ is the
closure of $I$ in the matroid. Therefore, the lattice of flats ${\mathcal
L}(\C)$ parametrizes subspaces in the linear code.
\end{rem}

We include the next result, which, despite being known, brings more details in
the spirit of the remark above; we present the proof for completeness.

\begin{prop}\label{C_I}
 Let $I$ be a flat of rank $k-r$ with $|I|=(k-r)+p_r-j$. Then $m(\C_I)=n-|I|=d_r(\C)+j$. Moreover, any $r$-dimensional subcode of minimal support is of the form $\C_I$ for some $I\subseteq [n]$ of rank $k-r$. In particular, $c_{r,p_r}$ counts the $r$-dimensional subcodes of minimal support.
\end{prop}
\begin{proof}
Let $\ker G_I={\rm Span} \{ v_1,\dots, v_r\}$. Clearly, $\supp (\C_I)\subseteq [n]\setminus I$. Assume that the inclusion is strict and pick an element, say $j\in ([n]\setminus I) \setminus\supp(\C_I)$. Let $w_i=v_iG$ appear as the $i^{\rm th}$ row of $G$. The $j^{\rm th}$ coordinate of $w_i$ is $(w_i)_j=v_iG_j=0$. This implies $\ker G_I=\ker G_{j\cup I}$ which contradicts $I$ being a flat. We conclude that the size of the support is exactly $n- |I|$.

The second claim is covered in the proof of \cite[Theorem 5.17]{JuPe}. Finally, $c_{r,p_r}$ gives the correct count of the subcodes of minimal support because every such subcode is of the form $\C_I$, where $|I|$ is maximal among all flats of the same rank and every such flat raises the coefficient by one.
\end{proof}

In our Example \ref{example0}, $c_{1,p_1}=c_{1,0}$ is the coefficient of $x$, and it is equal to $3$. There are exactly three codewords of minimum Hamming distance which correspond to each of the two rows and their difference.

\section{The degree of generalized star configuration schemes}

Let $\mathcal C=(\ell_1,\ldots,\ell_n)$ be a collection of $n$ linear forms in $R:=\mathbb K[x_1,\ldots,x_k]$, as at the beginning of Section 1. In what follows, we shed on some light on the primary decomposition of $I_a(\C)$ and use the result to compute its degree. All minimal primes of $I_a(\mathcal C)$ are ideals of the form $\langle \ell_{i_1},\ldots,\ell_{i_{n-a+1}}\rangle$ (see Section 2 in \cite{To1}). For clarifications, these minimal prime ideals are not minimally generated by those $n-a+1$ linear forms; in fact what makes the problem difficult and worth studying is the general situation when these minimal prime ideals are minimally generated by strict subsets of these $n-a+1$ linear forms.

Given a prime ${\mathfrak {p}}$ in $R$, we define the height (or codimension) of ${\mathfrak {p}}$ , written ${\rm ht}({\mathfrak {p}})$, to be the supremum of the lengths of all chains of prime ideals contained in ${\mathfrak {p}}$.

We are going to use the following notational conventions:

\begin{enumerate}
  \item If $I\subseteq [n]$, denote $\mathfrak p_I$ the prime ideal generated by all $\ell_i, i\in I$.
  \item If $\mathfrak i$ is an ideal, denote $\rho(\mathfrak i):=\{i,\ell_i\in\mathfrak i\}\subseteq [n]$.
  \item If $\mathfrak i$ is an ideal, denote $\nu(\mathfrak i):=|\rho(\mathfrak i)|$.
  \item We have $V({\mathfrak p_I})=\ker G_I$ and therefore, ${\rm ht} ({\mathfrak p_I})=\rank(G_I)=r(I)$.
\end{enumerate}

It is clear that $\rho(\mathfrak p_I)=cl(I)$, and $\mathfrak p_I=\mathfrak p_{cl(I)}$. Also, if $I,J\subseteq [n]$ are such that $\mathfrak p_I=\mathfrak p_J$, then $cl(I)=cl(J)$. Also, if $\mathfrak q$ is prime ideal, then $\mathfrak p_{\rho(\mathfrak q)}=\mathfrak q$.

\begin{rem} \label{connection} Suppose ${\rm ht}(I_a(\C))=k-r$, for some $0\leq r\leq k-1$.

If $\mathfrak p$ is a minimal prime over $I_a(\C)$, of height $k-r$, then from above $\mathfrak p= \langle \ell_{i_1},\ldots,\ell_{i_{n-a+1}}\rangle$ and we write $\mathfrak p=\mathfrak p_J,$ where $J=\{i_1,\ldots,i_{n-a+1}\}$. Furthermore, the rank of $J$ in the matroid is $r(J)=k-r$. Conversely, if $I\subseteq [n]$ is of rank $r(I)=k-r$, and with $|I|\geq n-a+1$, then $\mathfrak p_I$ is a prime ideal of height $k-r$ containing $I_a(\C)$. Since ${\rm ht}(\mathfrak p_I)=k-r={\rm ht}(I_a(\C))$, then $\mathfrak p_I$ is also minimal over $I_a(\C)$.
\end{rem}

\medskip

The first result of this section is similar to \cite[Exercise 3.25]{DePe}, and concerns the heights of the ideals $I_a(\C)$. Recall that $d_0(\C)=0$ by convention, and $d_k(\C)=n$ as we do not allow loops. Observe that the result of De Boer and Pellikaan is the case $r=0$ below.

\begin{prop}\label{height}
    For any $r=0,\ldots,k-1$, if $d_{r}(\C)<a\leq d_{r+1}(\C)$ , then $${\rm ht}(I_a(\C))=k-r.$$ In particular, $d_r(\mathcal C)=\max\{a:{\rm ht}(I_a(\mathcal C))=k-r+1\}$.
\end{prop}
\begin{proof}
For brevity, we will omit referring to $\C$ in our notations.

It is clear that if $1\leq a\leq a'\leq n$, then $I_{a'}\subset I_a$, and hence ${\rm ht}(I_{a'})\leq {\rm ht}(I_a)$. So the proof comes down to showing that for any $r=0,\ldots, k-1$, we have ${\rm ht}(I_{d_r+1})={\rm ht}(I_{d_{r+1}})=k-r$.

The height of $I_{d_r+1}$ is the minimum of the heights of its minimal primes. Any such minimal prime is generated by $n-(d_r+1)+1=n-d_r$ linear forms. Let $\mathfrak p$ be a minimal prime of minimum height. First, let us observe that ${\rm ht}(\frak p)\geq k-r$. If by contradiction, ${\rm ht}(\mathfrak p)= k-r-u$, for some $u\geq 1$, then in the matroid, the rank of the index set $J:=\rho(\mathfrak p)$ is $r(J)= k-(r+u)$. However, by Corollary \ref{genhamming}, this means that $n-d_{r+u}\geq |J|\geq n-d_r$. This gives $d_r\geq d_{r+u}$ which contradicts Theorem \ref{thm:We} (2). On the other hand, the same corollary shows that there is $J\subseteq [n]$ of rank $r(J)=k-r$ and $|J|=n-d_r$. Then, the prime ideal $\mathfrak p_J$ will contain $I_{d_r+1}$ and it has height $k-r$. So ${\rm ht}(I_{d_r+1})=k-r$.

It remains to show that ${\rm ht}(I_{d_{r+1}})=k-r$. The inequality $\leq$ is immediate since $I_{d_r+1}\supseteq I_{d_{r+1}}$. For the reverse inequality, let us take a minimial prime of $I_{d_{r+1}}$ and show that its height is $\geq k-r$. Let $\mathfrak q$ be one such minimal prime, and assume by contradiction that ${\rm ht}(\mathfrak q)=k-r-v$, for some $v\geq 1$. In the matroid, the rank of the index set $J':=\rho(\mathfrak q)$ is $r(J')= k-(r+v)$, hence, by Corollary \ref{genhamming}, $n-d_{r+v}\geq |J'|\geq n-d_{r+1}+1$. This leads to the contradiction $d_{r+1}-1\geq d_{r+v}$, for some $v\geq 1$.
\end{proof}

In Example \ref{example0}, our ideals are
\[
 I_1(\C)=\langle x_1, x_2\rangle, \, I_2(\C)=\langle x_1x_2, x_1(x_1+x_2),x_2(x_1+x_2)\rangle=\langle x_1,x_2\rangle^2, \, I_3(\C)=\langle x_1x_2(x_1+x_2)\rangle,
\]
with heights equal to $2, 2, 1$, respectively.

\medskip

Next we are interested in the primary decomposition of $I_a(\C)$. We denote ${\rm Min}_u(I_a(\C))$ to be the set of all distinct minimal prime ideals over $I_a(\C)$ of height $u$. The next result presents the primary components of lowest height of $I_a(\C)$, which are always unique in any minimal primary decomposition of $I_a(\C)$ (see \cite[Theorem 4.29]{Sh}).

\begin{prop}\label{primarydecomp} Let $r\in\{0,\ldots,k-1\}$, and let $d_{r}(\C)<a\leq d_{r+1}(\C)$. Then, $$I_a(\C)=\mathfrak p_1^{a-n+\nu(\mathfrak p_1)}\cap\cdots\cap \mathfrak p_s^{a-n+\nu(\mathfrak p_s)}\cap K,$$ where ${\rm Min}_{k-r}(I_a(\C))=\{\mathfrak p_1,\ldots,\mathfrak p_s\}$, and $K$ is an ideal of height $>k-r$.
\end{prop}
\begin{proof} We fix $r\in\{0,\ldots,k-1\}$ and let $a=d_{r}(\C)+j\leq d_{r+1}(\C)$, for some $1\leq j\leq d_{r+1}(\C)-d_{r}(\C)$. So ${\rm ht}(I_a(\C))=k-r$.

Then, a partial primary decomposition of $I_a(\C)$ is of the form $I_a(\C)=\mathfrak q_1\cap\cdots\cap \mathfrak q_s\cap K$, where all $\mathfrak q_i$ are $\mathfrak p_i$-primary ideals with $\mathfrak p_i\in {\rm Min}_{k-r}(I_a(\C))$, and $K$ is an ideal of height $>k-r$.

If $\mathfrak p$ is a minimal prime over $I_a(\C)$ of height $k-r$, then we know that $\mathfrak p=\mathfrak p_I$, for some $I\subset [n]$ with $|I|=n-a+1$, and the rank $r(I)=k-r$. It is clear that $\nu(\mathfrak p)=|cl(I)|\geq n-a+1\geq k-r$.

After some reordering, we assume that $cl(I)=\{1,\ldots,\nu(\mathfrak p)\}$, and we can assume that $\mathfrak p=\langle x_1,\ldots,x_{k-r}\rangle$. Then, in the localization $R_{\mathfrak p}$, the linear forms $\ell_{\nu(\mathfrak p)+1},\ldots,\ell_n$, are invertible, and because of this, in $R_{\mathfrak p}$, $$I_a(\C)R_{\mathfrak p}=I_{a-n+\nu(\mathfrak p)}(\tilde{\C})R_{\mathfrak p},$$ where $\tilde{\C}:=(\ell_1,\ldots,\ell_{\nu(\mathfrak p)})\subset \tilde{R}:=\mathbb K[x_1,\ldots,x_{k-r}]$.

We have that $\tilde{\C}$ is an $[\nu(\mathfrak p),k-r,\tilde{d}]$-linear code. Then $\nu(\mathfrak p)-\tilde{d}$ is the maximum number of linear forms of $\tilde{\C}$ that generate a height $(k-r)-1=k-r-1$ ideal. By Corollary \ref{genhamming}, $$\nu(\mathfrak p)-\tilde{d}\leq n-d_{r+1}(\C).$$ Since $a=d_{r}(\C)+j\leq d_{r+1}(\C)$, we obtain $$a-n+\nu(\mathfrak p)\leq \tilde{d}.$$ So, from \cite[Theorem 3.1]{To1}, we have $$I_{a-n+\nu(\mathfrak p)}(\tilde{\C})=\mathfrak p^{a-n+\nu(\mathfrak p)}.$$ So each $\mathfrak q_i$ is of the desired form $\mathfrak p_i^{a-n+\nu(\mathfrak p_i)}$.
\end{proof}

\begin{rem} Note that if $r=0$, then ${\rm Min}_k(I_a(\C))=\{\mathfrak m\}$, and $\nu(\mathfrak m)=n$. Accordingly, \cite[Theorem 3.1]{To1} gives $I_a(\C)=\mathfrak m^a$. If $r=1$ and $j=1$, then \cite[Lemma 2.2]{To1} gives $\mathfrak q_i=\mathfrak p_i$ and that $s$ is the number of projective codewords of minimum weight. Our Proposition \ref{primarydecomp} generalizes these results.

\medskip

In Example \ref{example0}, if $0=d_0(\C)<a\leq d_1(\C)=2$, indeed we have $I_a(\C)=\langle x_1,x_2\rangle^a$. If $2=d_1(\C)<a\leq d_2(\C)=3$, then $a=3$, and $k-r=2-1=1$. If $\mathfrak p\in {\rm Min}_1(I_3(\C))$, then $\nu(\mathfrak p)\geq n-a+1=1$, and $\mathfrak p=\mathfrak p_I$, for some $I\subseteq [n]$ with $r(I)=1$. Then $I$ can be only $\{1\},\{2\},\{3\}$ (hence we also have $I=cl(I)$, so $\nu(\mathfrak p)=1$). So $$I_3(\C)=\langle x_1\rangle\cap\langle x_2\rangle\cap\langle x_1+x_2\rangle\cap K,\, {\rm ht}(K)>1.$$ This indeed matches with $I_3(\C)=\langle x_1x_2(x_1+x_2)\rangle$; we can take the ideal $K$ to be $\langle x_1,x_2\rangle^3$, or any ideal primary to the maximal ideal $\langle x_1,x_2\rangle$, that contains $\langle x_1\rangle\cap\langle x_2\rangle\cap\langle x_1+x_2\rangle$.

\medskip

We make note here that even when we take $a=d_1(\C)+1$, in general there is little known about the ideal $K$. In a particular case useful in coding theory, the proof of \cite[Theorem 3.1]{AnTo} gives some information about $K$, namely the maximal degree of a generator of $K$ that is not in the saturation of $I_a(\C)$ with respect to $\mathfrak{m}$.
\end{rem}

\subsection{The degree of \texorpdfstring{$I_a(\C)$}.} Let us denote the Hilbert polynomial of the projective space ${\mathbb P}^m$ by $P_m(t)=\binom{t+m}{m}$.

Let $I$ be any homogeneous ideal in $R:=\mathbb K[x_1,\ldots,x_k]$, of height $c\leq k-1$. Then the Hilbert polynomial of $R/I$ is $$HP(R/I,t)=\deg(I)P_{k-c-1}(t)+\Omega_{k-c-1}(t),$$ where $\Omega_{k-c-1}(t)$ is some polynomial of degree $< k-c-1$. In other words, $$HP(R/I,t)=\frac{\deg(I)}{(k-c-1)!}t^{k-c-1}+\mbox{ lower degree terms}.$$

If $c=k$, then $R/I$ is an $R$-module of finite length, and by definition $\deg(I)=\dim_{\mathbb K}(R/I)<\infty$; note that in this case, the Hilbert polynomial is 0.

The next three lemmas are standard; we present them for completeness.

\begin{lem}\label{HilbPoly} Let $I,J,K$ be homogeneous ideals in $R=\mathbb K[x_1,\ldots,x_k]$, with $I=J\cap K$ and ${\rm ht}(I)={\rm ht}(J)\leq k-1$. Then
\begin{enumerate}
  \item If ${\rm ht}(K)>{\rm ht}(J)$, then $\deg(I)=\deg(J)$.
  \item If ${\rm ht}(K)={\rm ht}(J)$, but ${\rm ht}(J+K)>{\rm ht}(I)$, then $\deg(I)=\deg(J)+\deg(K)$.
\end{enumerate}
\end{lem}
\begin{proof} Computing the corresponding Hilbert polynomials through the classical short exact sequence
 \[
  0\to \frac{R}{J\cap K}\to \frac{R}{J}\oplus \frac{R}{K}\to \frac{R}{J+K}\to 0
 \] and comparing leading terms, one obtains the two desired conclusions.
\end{proof}

\begin{lem}\label{HPpower}
 Let $I$ be a linear ideal of height $c\leq k-1$ (i.e. generated by $c$ linearly independent forms) in $R=\K[x_1,\dots, x_k]$, then for any $i\geq 1$
 \[
  HP(R/I^i,t)=\binom{c+i-1}{c}P_{k-c-1}(t)+\Omega_{k-c-1}(t).
 \]
\end{lem}
\begin{proof}
 After a linear change of variables, we can assume that $I=\langle x_1,\dots,x_c\rangle$. So, $I^i$ is generated by all monomials in the first $c$ variables of total degree $i$. In the quotient, we have
 \[
  \frac{R}{I^i}\simeq (\K[x_1,\dots,x_c])_{\leq i-1}\otimes_{\K} \K[x_{c+1},\dots,x_k],
 \]
as graded vector spaces which leads to the above formula as $\dim  (\K[x_1,\dots,x_c])_{\leq i-1}=\binom{c+i-1}{c}$.
\end{proof}

\begin{lem}\label{degreeFatSchemes} Suppose $\mathfrak p_1,\ldots,\mathfrak p_s$ are distinct homogeneous prime ideals in $R=\mathbb K[x_1,\ldots,x_k]$ of the same height $c\leq k-1$. Let $n_1,\ldots,n_s\geq 1$ be integers. Then $$\deg(\mathfrak p_1^{n_1}\cap\cdots\cap\mathfrak p_s^{n_s})=\sum_{i=1}^s\deg(\mathfrak p_i^{n_i}).$$
\end{lem}
\begin{proof} We use induction on $s\geq 1$, with the $s = 1$ case being clear.

Let $s>1$, and let us denote $J:=\mathfrak p_1^{n_1}$ and $K:=\mathfrak p_2^{n_2}\cap\cdots\cap\mathfrak p_s^{n_s}$, and $I:=J\cap K$. If we show that ${\rm ht}(J+K)>c$, then with Lemma \ref{HilbPoly} (2) and from induction hypotheses we will be done.

Suppose by contradiction that ${\rm ht}(J+K)=c$. Since ${\rm ht}(J+K)={\rm ht}(\sqrt{J+K})$, and since $\sqrt{J+K}=\sqrt{\sqrt{J}+\sqrt{K}}$, we have $${\rm ht}(\mathfrak p_1+\mathfrak p_2\cap\cdots\cap\mathfrak p_s)=c.$$ Since $\mathfrak p_1\subseteq \mathfrak p_1+\mathfrak p_2\cap\cdots\cap\mathfrak p_s$, with ${\rm ht}(\mathfrak p_1)=c$, we must have $\mathfrak p_1=\mathfrak p_1+\mathfrak p_2\cap\cdots\cap\mathfrak p_s$, and so $$\mathfrak p_2\cap\cdots\cap\mathfrak p_s\subseteq\mathfrak p_1.$$ Hence some $\mathfrak p_i \subseteq \mathfrak p_1$, contradicting that they are distinct and of the same height.
\end{proof}

Recall that $c_{i,j}$ is the coefficient of $x^iy^j$ in $T_{\mathcal C}(x+1,y)$ and $p_r=\max \{ j: c_{r,j}\neq 0\}$. The main result of these notes is the following theorem.

\begin{thm}\label{main} Let $\C$ be an $[n,k]$-linear code in $R:=\mathbb K[x_1,\ldots,x_k]$. Let $a=d_r(\C)+j$, where $r=0,\dots, k$ and $j=1,\dots, d_{r+1}(\C)-d_r(\C)$. Then the degree of $I_a(\C)$ is determined by the coefficients of the Tutte polynomial:
\[
 \deg(I_a(\C))=\sum_{t=0}^{j-1} c_{r,p_{r}-t}.
\]
\end{thm}

\begin{proof} First suppose $r\geq 1$. From Proposition \ref{primarydecomp}, and from Lemma \ref{HilbPoly} (1), we have $$\deg(I_a(\C))=\deg(\mathfrak p_1^{a-n+\nu(\mathfrak p_1)}\cap\cdots\cap \mathfrak p_s^{a-n+\nu(\mathfrak p_s)}),$$ where $\{\mathfrak p_1,\ldots,\mathfrak p_s\}={\rm Min}_{k-r}(I_a(\C))$.

Lemma \ref{degreeFatSchemes} gives that $$\deg(I_a(\C))=\sum_{i=1}^s\deg(\mathfrak p_i^{a-n+\nu(\mathfrak p_i)}),$$ which in turn, from Lemma \ref{HPpower}, gives $$\deg(I_a(\C))=\sum_{i=1}^s\binom{\nu(\mathfrak p_i)-n+a+k-r-1}{k-r}.$$ where $p_{r}=n-d_r(\C)-k+r$ is the number defined in Lemma \ref{lemma:Tutte}.

If $j=1$, then $a=d_r(\C)+1$. Let $\mathfrak p\in {\rm Min}_{k-r}(I_a(\C))$. Since $n-d_r(\C)\geq \nu(\mathfrak p)=|\rho(\mathfrak p)|\geq n-a+1=n-d_r(\C)$, then $\nu(\mathfrak p)=n-d_r(\C)=n-a+1$. In the above formula we have $\deg(I_a(\C))$ equals $s$, the number of minimal primes of $I_a(\C)$ of height $k-r$. Also we have $\rho(\mathfrak p)$ is a flat of rank $k-r$ of maximum possible size $n-d_r(\C)=n-a+1$. Remark \ref{connection}, paired with Proposition \ref{C_I} gives that the number of minimal primes of $I_a(\C)$ of height $k-r$ equals $c_{r,p_r}$.

We prove the claimed formula by induction on $1\leq j\leq d_{r+1}(\C)-d_r(\C)-1$. The base case $j=1$, has been shown above. We must show that
$$\deg(I_{a+1}(\C))-\deg(I_a(\C))=c_{r,p_r-j}.$$

The left-hand side equals $$\sum_{\mathfrak p\in {\rm Min}_{k-r}(I_{a+1}(\C))}\binom{\nu(\mathfrak p)-n+a+k-r}{k-r}- \sum_{\mathfrak p\in {\rm Min}_{k-r}(I_{a}(\C))}\binom{\nu(\mathfrak p)-n+a+k-r-1}{k-r}.$$ As $I_{a+1}(\C)\subset I_a(\C)$, and as ${\rm ht}(I_a(\C))={\rm ht}(I_{a+1}(\C))=k-r$, we have ${\rm Min}_{k-r}(I_{a}(\C))\subseteq {\rm Min}_{k-r}(I_{a+1}(\C))$.

In fact, if $\mathfrak p\in {\rm Min}_{k-r}(I_{a+1}(\C))\setminus {\rm Min}_{k-r}(I_{a}(\C))$, then $\nu(\mathfrak p)=n-a$. Indeed, $\mathfrak{p} \in {\rm Min}_{k-r}(I_{a+1}(\C))$ implies that $\nu(\mathfrak{p}) \geq n-(a+1)+1 = n-a$. If we further had $\nu(\mathfrak{p}) \geq n-a+1$, then by definition $|\rho(\mathfrak{p})| \geq n-a+1$. Since the rank of $\rho(\mathfrak{p})$ is $k-r$, we have by Remark \ref{connection} that $\mathfrak{p} = \mathfrak{p}_{\rho(\mathfrak{p})} \in {\rm Min}_{k-r}(I_a(\C))$.

So, the left-hand side becomes $$|{\rm Min}_{k-r}(I_{a+1}(\C))\setminus {\rm Min}_{k-r}(I_{a}(\C))|+\sum_{\mathfrak p\in {\rm Min}_{k-r}(I_{a}(\C))}\binom{\nu(\mathfrak p)-n+a+k-r-1}{k-r-1}.$$

After identifying $p_r-j=n-a-k+r$ and saving $x^r$, we have that $c_{r,p_r-j}$ is the coefficient of $y^{n-a-k+r}$ in $$\sum_{I\subseteq[n], r(I)=k-r}(y-1)^{|I|-k+r}.$$

We need to investigate those $I$'s of rank $k-r$ such that $|I|-k+r\geq n-a-k+r$, or simply $|I|\geq n-a$. In fact, if $|I|=n-a=|cl(I)|$ and $r(I)=k-r$, then $\langle \ell_i|i\in I\rangle\in {\rm Min}_{k-r}(I_{a+1}(\C))\setminus {\rm Min}_{k-r}(I_{a}(\C))$.

So the right-hand side of our claimed formula becomes $$|{\rm Min}_{k-r}(I_{a+1}(\C))\setminus {\rm Min}_{k-r}(I_{a}(\C))|+\Delta,$$ where $\Delta$ is the coefficient of $y^{n-a-k+r}$ in $$\sum_{\mathfrak p\in {\rm Min}_{k-r}(I_{a}(\C))}\left(\sum_{I\in \Lambda(\mathfrak p)}(y-1)^{|I|-k+r}\right),$$ where for $\mathfrak p\in {\rm Min}_{k-r}(I_{a}(\C))$, we define $\Lambda(\mathfrak p):=\{I\subseteq[n]:\mathfrak p=\langle \ell_i|i\in cl(I)\rangle\mbox{ and }|I|\geq n-a\}$.

Let $\mathfrak p=\langle \ell_{i_1},\ldots,\ell_{i_{\nu(\mathfrak p)}}\rangle\in {\rm Min}_{k-r}(I_{a}(\C))$, and let $I\subset \{i_1,\ldots,i_{\nu(\mathfrak p)}\}$, with $|I|\geq n-a$. The prime ideal $\langle \ell_i|i\in I\rangle$ contains the ideal $I_{a+1}(\C)$, whose height is $k-r$. So $r(I)=k-r$.

This gives that $\Delta$ is the coefficient of $y^{n-a-k+r}$ in $$\sum_{\mathfrak p\in {\rm Min}_{k-r}(I_{a}(\C))}\left(\sum_{u=n-a}^{\nu(\mathfrak p)}\binom{\nu(\mathfrak p)}{u}(y-1)^{u-k+r}\right).$$ This implies that it is enough to show that for each $\mathfrak p\in {\rm Min}_{k-r}(I_{a}(\C))$, one has $$\binom{\nu(\mathfrak p)-n+a+k-r-1}{k-r-1}=\sum_{u=n-a}^{\nu(\mathfrak p)}(-1)^{u-n+a}\binom{\nu(\mathfrak p)}{u}\binom{u-k+r}{u-n+a}.$$

Denote $\alpha:=\nu(\mathfrak p),\beta:=n-a,\gamma:=k-r$. Then we have
$\alpha>\beta\geq\gamma\geq 1$. By induction on pairs $(\beta,\gamma)$ with
$\beta \geq \gamma \geq 1$ we prove the desired formula
\begin{equation}\label{desired_formula}
  \binom{\alpha-\beta+\gamma-1}{\gamma-1}=\sum_{u=\beta}^{\alpha}(-1)^{u-\beta}\binom{\alpha}{u}\binom{u-\gamma}{u-\beta}.
\end{equation}

$\bullet$ If $\beta=\gamma$, then we have to show $$\binom{\alpha-1}{\beta-1}=\sum_{u=\beta}^{\alpha}(-1)^{u-\beta}\binom{\alpha}{u}.$$ But this is a well-known formula following from the fact that
\begin{eqnarray}
\binom{\alpha-1}{\beta-1}&=&\binom{\alpha}{\beta}-\binom{\alpha-1}{\beta}=\binom{\alpha}{\beta}-\left(\binom{\alpha}{\beta+1} -\binom{\alpha-1}{\beta+1}\right)\nonumber\\
&=&\binom{\alpha}{\beta}-\binom{\alpha}{\beta+1}+\left(\binom{\alpha}{\beta+2} -\binom{\alpha-1}{\beta+2}\right) =\cdots.\nonumber
\end{eqnarray}

$\bullet$ If $\gamma = 1$, then we must show
  \[
    S_{\beta} :=
    \sum_{u=\beta}^{\alpha}(-1)^{u-\beta}\binom{\alpha}{u}\binom{u-1}{u-\beta}
    = 1.
  \]
We prove this formula by induction on $\beta \geq 1$. The $\beta = 1$ case
gives $\beta = \gamma$, which was addressed previously. Now, assume the result
holds for $\beta \geq 1$. Then the inductive hypothesis yields
  \begin{align*}
    1 &=
    \sum_{u=\beta}^{\alpha}(-1)^{u-\beta}\binom{\alpha}{u}\binom{u-1}{u-\beta}
    \\
    &= \binom{\alpha}{\beta} - \sum_{u =
      \beta+1}^{\alpha}{(-1)^{u-\beta+1}\binom{\alpha}{u}\binom{u-1}{u-\beta}}
    \\
    &= \binom{\alpha}{\beta} - \sum_{u =
      \beta+1}^{\alpha}{(-1)^{u-\beta+1}\binom{\alpha}{u}\left(\binom{u}{u-\beta}-\binom{u-1}{u-\beta-1}\right)}
    \\
    &= \binom{\alpha}{\beta} - \sum_{u =
      \beta+1}^{\alpha}{(-1)^{u-\beta+1}\binom{\alpha}{u}\binom{u}{u-\beta}}
    + S_{\beta+1}.
  \end{align*}
Thus it suffices to show that
  \[
    \sum_{u =
      \beta+1}^{\alpha}{(-1)^{u-\beta+1}\binom{\alpha}{u}\binom{u}{u-\beta}} =
    \binom{\alpha}{\beta}.
  \]
Using the fact that $\binom{u}{u-\beta} = \binom{u}{\beta}$ together with the
binomial identity $\binom{\alpha}{u}\binom{u}{\beta} =
\binom{\alpha}{\beta}\binom{\alpha-\beta}{u-\beta}$, the equation above
reduces to
  \begin{align*}
    \sum_{u =\beta+1}^{\alpha}{(-1)^{u-\beta+1}\binom{\alpha-\beta}{u-\beta}} = 1,
  \end{align*}
which follows immediately from the fact that the alternating sum of binomial
coefficients vanishes.

$\bullet$ We are now ready to prove (\ref{desired_formula}) by induction. We
prove it is true for $\beta+1$, assuming it is true for any $\beta'$ with
$\gamma \leq \beta' \leq \beta$. First we
have
$$\binom{\alpha-(\beta+1)+\gamma-1}{\gamma-1}=\binom{\alpha-\beta+\gamma-2}{\gamma-1}=\binom{\alpha-\beta+\gamma-1}{\gamma-1}-
\binom{\alpha-\beta+\gamma-2}{\gamma-2}.$$
Next we use the induction hypothesis on each of the terms of the difference
above to obtain
\begin{eqnarray}
&=&\sum_{u=\beta}^{\alpha}(-1)^{u-\beta}\binom{\alpha}{u}\binom{u-\gamma}{u-\beta}-\sum_{u=\beta}^{\alpha}(-1)^{u-\beta} \binom{\alpha}{u}\binom{u-(\gamma-1)}{u-\beta}\nonumber\\
&=& \sum_{u=\beta}^{\alpha}(-1)^{u-\beta}\binom{\alpha}{u}\left(\binom{u-\gamma}{u-\beta} -\binom{u-(\gamma-1)}{u-\beta}\right).\nonumber
\end{eqnarray}
For $u=\beta$, the corresponding term in the sum is zero, so we can start the sum from $\beta+1$. We obtain $$=\sum_{u=\beta+1}^{\alpha}(-1)^{u-\beta}\binom{\alpha}{u}\left(-\binom{u-\gamma}{u-\beta-1}\right)= \sum_{u=\beta+1}^{\alpha}(-1)^{u-(\beta+1)}\binom{\alpha}{u}\binom{u-\gamma}{u-(\beta+1)}.$$ And we are done.

If $r=0$, then $I_a(\C)=\mathfrak m^a$ (see the note after the proof of Proposition \ref{primarydecomp}). Therefore, $\deg(I_a(\C))=\binom{k+a-1}{k}$. Then, $p_0=n-k$, and $c_{0,p_0-t}$ is the coefficient of $y^{n-k-t}$ in
\[
 \displaystyle \sum_{I\subset[n], r(I)=k}(y-1)^{|I|-t}.
\]
 Similar, yet simpler combinatorics as in the previous argument will prove the claimed formula for the case $r=0$.
\end{proof}

\begin{ex}(\cite{GrSt}) Consider the code $\C=(x_1,x_2,x_3,x_1+x_2,x_1-x_2,x_1+x_3,x_1-x_3,x_2+x_3,x_2-x_3)\subset\mathbb K[x_1,x_2,x_3]$ which corresponds to the $B_3$ root system, with $n=9$, $k=3$ and ${\rm char} (\K)\neq 2$. We have
\[
 T_{\C}(x+1,y)=y^{6}+3 y^{5}+6 y^{4}+x^{3}+ 3 x y^{2}+10 y^{3}+9 x^{2}+10 x y+15 y^{2}+23 x+18 y+15.
\] Evidently, $p_0=6, p_1=2, p_2=0$ and $p_3=0$. Therefore, the generalized Hamming weights are $d_0=0, d_1=9-2-3+1=5, d_2=9-0-3+2=8$ and $d_3=9-0+3-3=9$.

Analyzing the $a$-fold products of linear forms of $\C$, we obtain,
\begin{eqnarray*}
   &&\hi (I_1(\C)) = \hi (I_2(\C))=\hi (I_3(\C))=\hi (I_4(\C))=\hi (I_{\underline{\bf 5}}(\C))=3,\\
   &&\hi (I_6(\C)) =  \hi (I_7(\C))=\hi (I_{\underline{\bf 8}}(\C))=2,\\
   &&\hi (I_{\underline{\bf 9}}(\C)) =  1.
\end{eqnarray*}
The degrees and the Hilbert polynomials of the ideals are summarized below.
\[
\begin{array}{|c||c|c|c|c|c|}
\hline
 i & 1 & 2 & 3 & 4 & 5 \\
 \hline
 \deg(I_i(\C)) & 1 & 1+3=4 & 1+3+6=10 &  1+3+6+10=20 & 1+3+6+10+15=35\\
 \hline
\end{array}
\]
\[
\begin{array}{|c||c|c|c|c|}
\hline
 i & 6 & 7 & 8 & 9 \\
 \hline
 HP(R/I_i(\C)) & 3 P_0 & (3+10)P_0 & (3+10+23)P_0 & 9 P_1-36 P_0 \\
 \hline
\end{array}
\]
\end{ex}

Another homological invariant of interest is the minimum number of generators of $I_a(\C)$, often denoted by $\mu(I_a(\C))$. This number can be also read from the coefficients of the Tutte polynomial.

\begin{prop}\label{mingens} Let $\C$ be an $[n,k]-$linear code. Then, with the previous notations, for any $a\in\{1,\ldots,n\}$, one has $$\mu(I_a(\C))=\sum_{u=0}^{\min\{k,n-a\}}c_{k-u,n-a-u}.$$
\end{prop}
\begin{proof} The ideal $I_a(\C)$ is generated in degree $a$; therefore, the minimum number of generators is $\mu(I_a(\C))=\dim_{\mathbb K} (I_a(\C))_a$.

Berget, in \cite{Be}, constructs the following graded vector space. Suppose $\displaystyle \mathcal C=(\ell_1,\ldots,\ell_n)\subset R:=\mathbb K[x_1,\ldots,x_k]$. For any $I\subset [n]$, denote $\displaystyle\ell_I=\Pi_{i\in I}\ell_i$, with the convention that $\ell_{\emptyset}=1$. Let $P(\mathcal C)$ be the $\mathbb K$-vector subspace of $R$ spanned by $\ell_I$, for all $I\subset [n]$. Then one has a decomposition:
	\[
	P(\mathcal C)=\bigoplus_{0\leq u\leq v\leq n}P(\mathcal C)_{u,v},
	\]
	where
	\[
	P(\mathcal C)_{u,v}={\rm Span}_{\mathbb K}\{\ell_I|\dim_{\mathbb K}({\rm Span}_{\mathbb K}\{\ell_j,j\in[n]\setminus I\})=u\mbox{ and }v=n-|I|\}.
	\] We have $0\leq u\leq k$ since ${\rm Span}_{\mathbb K}\{\ell_j,j\in[n]\setminus I\})$ is a subspace of $\mathbb K^k$.
	
As observed in \cite[Remark 2.3]{GaTo}, $\displaystyle\bigoplus_{0\leq u\leq n-a}P(\mathcal C)_{u,n-a}=(I_a({\mathcal C}))_{a}$. Therefore $$\mu(I_a(\C))=\sum_{u=0}^{n-a}\dim_{\mathbb K}P(\mathcal C)_{u,n-a}.$$

The main result of \cite{Be} is Theorem 1.1 which says that $$T_{\C}(x,y)=\sum_{0\leq u\leq v\leq n}(x-1)^{k-u}y^{v-u}\dim_{\mathbb K}P(\C)_{u,v}.$$

In our notations, $\dim_{\mathbb K}P(\C)_{u,v}=c_{k-u,v-u}$, and hence the claimed formula.
\end{proof}

\begin{ex} Let $\C$ be the same as in the above example. The following calculations were checked in \cite{GrSt}.

\begin{eqnarray*}
\mu(I_1(\C))&=&c_{3,8}+c_{2,7}+c_{1,6}+c_{0,5}=0+0+0+3=3\\
\mu(I_2(\C))&=&c_{3,7}+c_{2,6}+c_{1,5}+c_{0,4}=0+0+0+6=6\\
\mu(I_3(\C))&=&c_{3,6}+c_{2,5}+c_{1,4}+c_{0,3}=0+0+0+10=10\\
\mu(I_4(\C))&=&c_{3,5}+c_{2,4}+c_{1,3}+c_{0,2}=0+0+0+15=15\\
\mu(I_5(\C))&=&c_{3,4}+c_{2,3}+c_{1,2}+c_{0,1}=0+0+3+18=21\\
\mu(I_6(\C))&=&c_{3,3}+c_{2,2}+c_{1,1}+c_{0,0}=0+0+10+15=25\\
\mu(I_7(\C))&=&c_{3,2}+c_{2,1}+c_{1,0}=0+0+23=23\\
\mu(I_8(\C))&=&c_{3,1}+c_{2,0}=0+9=9\\
\mu(I_9(\C))&=&c_{3,0}=1
\end{eqnarray*}
\end{ex}

\section{Appendix: A conjecture about the graded minimal free resolution of \texorpdfstring{$I_a(\C)$}.}

Throughout this section we assume that $\mathbb K$ is a field of characteristic 0. Let $\mathcal C$ be an $[n,k]-$linear code over $\mathbb K$, with generating matrix $G$, of size $k\times n$ and of rank $k$, and with no zero columns. Let $\ell_i$ be the linear form in $R:=\mathbb K[x_1,\ldots,x_k]$, dual to the column $i$ of $G$. As we did before, we will abuse notation by saying that $\C=(\ell_1,\ldots,\ell_n)\subset R$. Throughout this section the notation $\C\subset R$ will mean an $[n,k]-$linear code $\C$ satisfying these conditions. For $1\leq a\leq n$, let $I_a(\C)$ be the ideal generated by all $a-$fold products of the linear forms in $\C$.

Our ultimate goal in studying the homological properties of the ideals $\{I_a(\C)\}_a$ has been focused on the following conjecture:

\begin{conj}\label{conj1}
 For any $\mathcal C\subset R$ and any $1\leq a\leq n$, the ideal $I_a(\C)$ has a linear graded free resolution. Or equivalently, the Castelnuovo-Mumford regularity, ${\rm reg}(R/I_a(\C))$, equals $a-1$. (see \cite{EiGo}).
\end{conj}

The conjecture has been verified for any $\mathcal C$ and any $1\leq a\leq d_1(\C)$ (\cite[Theorem 3.1]{To1}), and whenever the ring $R/I_a(\C)$ is determinantal (see the argument following the proof of Proposition 2.1 in \cite{To5}; $I_a(\C)$ is generated by the maximal minors of an $a\times n$ matrix with linear forms entries, and in the mentioned conditions the Eagon-Northcott complex becomes a free resolution). For example, the defining ideals of (usual) star configurations fit under the latter case. The conjecture is also true for $a=d+1$, under some special conditions (see \cite[Theorem 3.1]{AnTo}).


The key step in the proof of \cite[Theorem 2.5]{GaTo} was to show that for (usual) star configurations built on $\mathcal C$,
\[
I_a(\C):\ell=I_{a-1}(\C'),
\]
for all $2\leq a\leq n$, where $\ell$ is any linear form in $\mathcal C$, and $\mathcal C'=\mathcal C\setminus\{\ell\}$. We conjecture that the above equality of ideals is true for any $\mathcal C$:

\begin{conj}\label{conj2} For any $\C\subset R$ and for any $\ell\in \C$, in $R$ one has the equality of ideals $I_a(\C):\ell=I_{a-1}(\C'), \mbox{ for all } 2\leq a\leq n$.
\end{conj}

For notation purposes, we will find it useful to make the convention $I_0(\C):=R$, for any linear code $\C$. We can extend the range of $a$ to $1,\ldots, n$, since $I_1(\C)=\mathfrak m$, and therefore $I_1(\C):\ell=R=I_0(\C')$.

If ${\rm ht}(I_{a-1}(\C'))=k$, then $I_{a-1}(\C')=\mathfrak m^{a-1}$ (by \cite[Theorem 3.1]{To1}). Since obviously $I_a(\C):\ell\subseteq \mathfrak m^{a-1}$, we see that Conjecture \ref{conj2} is satisfied for $a\in\{1,\ldots,d_1(\C)\}$.

\begin{lem}\label{coloop} Let $\C\subset R$. If $\ell$ is a coloop in $\M(\C)$, then in $R$ one has $$I_a(\C):\ell=I_{a-1}(\C'),$$ for all $a=1,\ldots,n$.
\end{lem}
\begin{proof} Assume $\ell=\ell_n=x_k$ is a coloop. Then the last row of $G$,
  the generating matrix of $\mathcal C$, has 1 in the last entry and 0
  everywhere else. Hence $\mathcal C'\subset S:=\mathbb
  K[x_1,\ldots,x_{k-1}]$. Then for all $b=1,\ldots,n-1$, the generators of
  $I_b(\C')$ are elements of $S:=\mathbb K[x_1,\ldots,x_{k-1}]\subset R$. Because of this inclusion, the ideal $I_b(\C')R$ of $R$ will be written simply as $I_b(\C')$.

  If $f\in I_a(\C):x_k$, then $x_kf\in I_a(\C)=x_kI_{a-1}(\C')+I_a(\C')$ shows that
  $x_k(f-g)\in I_a(\C')$ for some $g\in I_{a-1}(\C')$.

Denote $h:=f-g$. Then we can write $$h=x_k^mh_m+\cdots+x_kh_1+h_0, h_i\in S,$$ for some $m\leq \deg(h)$.

Because the generators of $I_a(\C')$ are in $S$, taking partial derivative $\partial_k=\partial/\partial x_k$ of $x_kh\in I_a(\C')$, we obtain
\[
h+x_k\partial_kh\in I_a(\C').
\]
 Multiplying this by $x_k$, we obtain $x_k^2\partial_kh\in I_a(\C')$.
 Again, taking the partial derivative w.r.t. $x_k$ and multiplying by $x_k$, we obtain
 \[
 x_k^3\partial_k^2h\in I_a(\C').
 \]
 Repeating the same argument, we see that $x_{k}^{m+1}\partial_k^m h\in I_a(\C')$ but $\partial_k^m h=m! h_m$. So, we showed that $x_{k}^{m+1}h_m\in I_a(\C')$.

Taking partial derivative w.r.t. $x_k$, leads to $$x_k^mh_m\in I_a(\C').$$ With the same trick applied to $h-x_k^mh_m$, we recursively show that all $x_k^ih_i$'s, and hence $h$ itself, belong to $I_a(\C')\subset I_{a-1}(\C')$. Therefore $f\in I_{a-1}(\C')$, leading to $I_a(\C):\ell=I_{a-1}(\C')$.
\end{proof}

\medskip

Similar to the assumption at the beginning of the proof of Lemma \ref{coloop}, after a change of variables and a possible permutation of the columns of $G$, we can assume the $\ell$ that shows up in the arguments below is $\ell=\ell_n=x_k$. Then we consider $\tilde{\mathcal C}$ to be the linear code with generating matrix $\tilde{G}$ obtained from $G$ by deleting the last row and the last column. As the anonymous referee pointed out, the matrix $\tilde{G}$ may have zero columns; assume that we have exactly $\tilde{n}_0$ such columns.

Let $\mathcal C''$ be the linear code with generating matrix $G''$ obtained from $\tilde{G}$ after removing all those zero columns. Then $\mathcal C''$ is an $[n-1-\tilde{n}_0,k-1]-$linear code, with defining linear forms in $S:=\mathbb K[x_1,\ldots,x_{k-1}]$.

\begin{rem}\label{pathologic}
(a) If $a\leq n-1-\tilde{n}_0$, then $$\langle x_k,I_a(\C)\rangle=\langle x_k,I_a(\mathcal C'')\rangle.$$

(b) If $a\geq n-\tilde{n}_0$, then every generator of $I_a(\C)$ will have a factor of $x_k$, since $G$ has exactly $\tilde{n}_0+1$ columns proportional to $x_k$. Then in this case $I_a(\C)=x_k I_{a-1}(\C')$, where $\C'=\C\setminus\{x_k\}$, hence Conjecture \ref{conj2} is superfluously satisfied when $n\geq a\geq n-\tilde{n}_0$. Also $\langle x_k,I_a(\C)\rangle=\langle x_k\rangle$ which in turn will equal to $\langle x_k,I_a(\C'')\rangle$, if we make the convention that $I_b(\mathcal D)=0$ for any $b>|\mathcal D|$, where $\mathcal D$ is any linear code of block-length $|\mathcal D|$, and with no zero column in the generating matrix.
\end{rem}

\begin{rem}\label{Berget} We used the construction of Berget \cite{Be} in the proof of Proposition \ref{mingens} and invoked \cite[Remark 2.3]{GaTo}:
\[
\displaystyle\bigoplus_{0\leq u\leq n-a}P(\mathcal C)_{u,n-a}=(I_a({\mathcal C}))_{a}.
\]
This gives the degree $a$ piece of $P(\mathcal C)$ (see the beginning of \cite[Section 2]{Be}). So by \cite[Lemma 6.2]{Be}, if $\ell$ is not a coloop of the matroid of $\mathcal C$, then
	\[
	0\to P(\mathcal C')(-1)\xrightarrow{\cdot \ell} P(\mathcal C)\to P(\mathcal C'')\to 0.
	\]
    In degree $a$, the short exact sequence gives the following formula for dimensions: $$\dim P(\mathcal C')_{a-1}=\dim P(\mathcal C)_a-\dim P(\mathcal C'')_a.$$ This, combined with the identification above and with the standard short exact sequence of $R$-graded modules
\[
0\longrightarrow \frac{R(-1)}{I_a(\C):\ell}\stackrel{\cdot \ell}\longrightarrow \frac{R}{I_a(\C)}\longrightarrow \frac{R}{\langle \ell,I_a(\mathcal C'')\rangle}\longrightarrow 0,
\] leads to $\dim(I_a(\C):\ell)_{a-1}=\dim (I_{a-1}(\C'))_{a-1},$ and hence the equality $$(I_a(\C):\ell)_{a-1}=(I_{a-1}(\C'))_{a-1}.$$
\end{rem}

\begin{rem}\label{mingens_conj2} The ideal $I_a(\C):\ell$ is minimally generated in degree $\geq a-1$, since $I_a(\C)$ is generated in degree $a$. We have that $I_{a-1}(\C')\subset I_a(\C):\ell$. Therefore Conjecture \ref{conj2} is equivalent to showing that for any $[n,k]-$linear code $\C$ and any $\ell\in \C$ one has the equality of minimum number of generators $\mu(I_a(\C):\ell)=\mu(I_{a-1}(\C'))$, for all $1\leq a\leq n$.
\end{rem}

\begin{prop}\label{regColon} Conjecture \ref{conj2} is equivalent to the following two conditions combined:
\begin{enumerate}
 \item Conjecture \ref{conj1}, and
 \item for any $[n,k]-$linear code $\C$ and any $\ell\in \C$ we have $\dim_{\mathbb K}(I_a(\C):\ell)_a=\dim_{\mathbb K}(I_{a-1}(\C'))_a$, for all $1\leq a\leq n$.
\end{enumerate}
\end{prop}
\begin{proof} Let $\ell\in \C$. In $R$ we have $I_a(\C)=\ell I_{a-1}(\C')+I_a(\mathcal C')$. Assume $\ell=\ell_n=x_k$.	
	
Before moving further, we analyze the extreme cases separately:
\begin{itemize}
	 \item If $a=n$, then $I_n(\mathcal C)$ is a principal ideal generated by $\ell_1\cdots\ell_n$, so ${\rm reg}(R/\langle \ell_1\cdots\ell_n\rangle)=n-1$, and $I_n(\mathcal C):\ell_n=\langle \ell_1\cdots\ell_{n-1}\rangle=I_{n-1}(\mathcal C')$. So in order to make sense of $I_a(\mathcal C'')$, we can assume $1\leq a\leq n-1$.
	
	 \item If $a=1$, then $I_1(\mathcal C)=\langle \ell_1,\ldots,\ell_n\rangle=\mathfrak m$, leading to ${\rm reg}(R/I_1(\mathcal C))=0$, and to $I_1(\mathcal C):\ell=R=I_0(\mathcal C')$, by the convention.
\end{itemize}
	
We have the standard short exact sequence of $R-$graded modules
\[
0\longrightarrow \frac{R(-1)}{I_a(\C):x_k}\stackrel{\cdot x_k}\longrightarrow \frac{R}{I_a(\C)}\longrightarrow \frac{R}{\langle x_k,I_a(\mathcal C'')\rangle}\longrightarrow 0.
\]

Let us denote
\[
r':={\rm reg}(R/(I_a(\C):x_k)),\, r:={\rm reg}(R/I_a(\C)),\, r'':={\rm reg}(R/\langle x_k,I_a(\C'')\rangle).
\]

For any ideal $J$ of $R$, one has ${\rm reg}(R(-1)/J)=1+{\rm reg}(R/J)$. This can be easily seen by analyzing the definition of regularity using the shifts in the graded minimal free resolution.

From the classical inequalities for the regularity under short exact sequences \cite[Corollary 20.19]{Ei0}, one has	

\begin{eqnarray*}
    r'+1&\leq&\max\{r,r''+1\},\\
    r&\leq&\max\{r'+1,r''\},\\
    r''&\leq&\max\{r',r\}.
\end{eqnarray*}

Now we show $\Rightarrow$. If we are in the situation (b) in Remark \ref{pathologic}, then $r''=0$. The first two inequalities above give that either $r=0$, or $r=r'+1$. So, $I_a(\C)$ is generated by linear forms (namely $I_a(\C)=\frak m$, or equivalently $a=1$), or ${\rm reg}(R/I_a(\C))=a-1$ if and only if ${\rm reg}(R/I_{a-1}(\C'))=a-2$. This implies that we can keep removing $x_k$ from $\C$ until we are in situation (a) in Remark \ref{pathologic} when we have
\[
r'={\rm reg}(R/I_{a-1}(\C'))=a-2
\]
and
\[
r''={\rm reg}(S/I_a(\mathcal C''))=a-1.
\]
For the above equalities we used induction on $n\geq 2$, and the assumption that Conjecture \ref{conj2} is valid for $\C,\C',$ and $\C''$. The base case of the induction,i.e. $n=2$, is treated at the beginning of the proof: $a$ can only be equal to either $2=n$ or $1$.

The last two inequalities above become
\begin{eqnarray*}
    r&\leq&\max\{a-1,a-1\}\\
  a-1 & \leq & \max \{ a-2,r\},
\end{eqnarray*} which implies the desired formula $r={\rm reg}(R/I_a(\C))=a-1$.

Part (2) is immediate from Conjecture \ref{conj2} being true.

\medskip

For the converse implication $\Leftarrow$, suppose $x_k$ is not a coloop; the case when $x_k$ is a coloop is dealt with by Lemma \ref{coloop}. Also we may assume we are in situation (a) in Remark \ref{pathologic}, since situation (b) gives Conjecture \ref{conj2} for free.

If we assume Conjecture \ref{conj1} to be true, then $r=a-1$ and $r''=a-1$, and the three inequalities of regularity become:
\begin{eqnarray*}
    r'+1&\leq&\max\{a-1,a\},\\
    a-1&\leq&\max\{r'+1,a-1\},\\
    a-1&\leq&\max\{r',a-1\}.
\end{eqnarray*} The first inequality clearly gives $r'\leq a-1$. This means that $I_a(\C):x_k$ is minimally generated in degrees $a-1$ and (possibly) $a$. Remark \ref{Berget} gives that $(I_a(\C):x_k)_{a-1}=(I_{a-1}(\C'))_{a-1}$, whereas condition (2) gives $(I_a(\C):x_k)_{a}=(I_{a-1}(\C'))_{a}$. So $I_a(\C):x_k$ is in fact minimally generated in degree $a-1$. Remark \ref{mingens_conj2} will then prove the claim.
\end{proof}

\medskip

In order to have $I_a(\C):\ell=I_{a-1}(\C')$, one must at least have the equality up to radicals, which we verify now.

\begin{prop}\label{prop:radicalColon} For any $a=2,\ldots,n$, and any $\ell\in\mathcal C$, we have: $$\sqrt{I_a(\C):\ell}=\sqrt{I_{a-1}(\C')},$$ where $\mathcal C'=\mathcal C\setminus\{\ell\}$.
\end{prop}
\begin{proof} Assume $\ell=\ell_n=x_k$. It is enough to prove that $I_a(\C):x_k\subseteq \sqrt{I_{a-1}(\C')}$.

Suppose not. Then there exists $f\in I_a(\C):x_k$, and there exists $\mathfrak q$ a minimal prime over $I_{a-1}(\C')$ with $f\notin\mathfrak q$.

We have $x_kf\in I_a(\C)\subset I_{a-1}(\C')\subset \mathfrak q$. So $x_k\in \mathfrak q$, as otherwise we would have $f\in \mathfrak q$.

Taking partial derivative of $x_kf$ w.r.t. $x_k$, one obtains $$f+x_k\partial_kf\in I_{a-1}(\mathcal C),$$ and hence $f\in \langle x_k\rangle+I_{a-1}(\mathcal C)$.

Since $$I_{a-1}(\mathcal C)\subset \langle x_k\rangle+ I_{a-1}(\C')\subset \langle x_k\rangle+\mathfrak q=\mathfrak q,$$ one obtains that $f\in \mathfrak q$. Contradiction.
\end{proof}

\medskip

Let $\C=(\ell_1,\ldots,\ell_n)\subset R:=\K[x_1,\ldots,x_k]$ be a collection of linear forms, some possibly proportional. Suppose $\langle \ell_1,\ldots,\ell_n\rangle=\langle x_1,\ldots,x_k\rangle=:\frak m$. Denote also with $\C$ the $[n,k]-$linear code whose generating matrix $G$ has columns dual to the linear forms $\ell_1,\ldots,\ell_n$. Also let $\C'=\C\setminus\{\ell\}$, where $\ell\in\C$.

Let $p_r$, and $p_r'$ denote the numbers showing up in Lemma \ref{lemma:Tutte} corresponding to $\C$ and, respectively, to $\C'$. Then, in support of Conjecture \ref{conj2} we have the following result.

\begin{prop}\label{recdegree} Let $a=d_r(\C)+j\leq d_{r+1}(\C)$ and $\ell\in\C$. If $j\geq 2$, or if $j=1$ and $p_r=p_r'$, then

\[
\deg(I_a(\C):\ell)=\deg(I_{a-1}(\C')).
\]
\end{prop}
\begin{proof} Because of Lemma \ref{coloop}, we can assume that $\ell$ is not a coloop which means that $\dim(\C')=k$.

Let $m=\max \{|J|:r(J)=k-r\}$ and $m'=\max \{|J|:r'(J)=k-r\}$. From Corollary \ref{genhamming}, we know that $d_r(\C)=n-m$ and $d_r(\C')=n-m'$. Clearly, we have $m'\leq m\leq m'+1$ which implies that $$d_r(\C')\leq d_r(\C)\leq d_r(\C')+1.$$

Suppose $a=d_r(\C)+j$, where $j\in\{1,\ldots,d_{r+1}(\C)-d_r(\C)\}$. So ${\rm ht}(I_a(\C))=k-r$. Since $I_a(\C)\subset I_{a-1}(\C')$, then ${\rm ht}(I_a(\C))\leq {\rm ht}(I_{a-1}(\C'))$. The inequality is strict exactly when $a-1=d_r(\C)+j-1\leq d_r(\C')$ which is only possible when $j=1$ and $d_r(\C)=d_r(\C')$. Let us analyze this situation. Let $a=d_r(\C)+1$ and consider the decomposition obtained in Proposition \ref{primarydecomp}
\[
 I_a(\C)=\mathfrak p_1\cap\cdots\cap \mathfrak p_s\cap K,
\]
where ${\rm Min}_{k-r}(I_a(\C))=\{\mathfrak p_1,\ldots,\mathfrak p_s\}$, and $K$ is an
ideal of height $>k-r$. The reason why the powers of $\mathfrak p_i$ are all 1 is explained at the beginning of the proof of Theorem \ref{main}, when we analyzed the case $j=1$.

Coloning this by $\ell$, we have $$I_a(\C):\ell=(\mathfrak p_1:\ell)\cap\cdots\cap
(\mathfrak p_s:\ell)\cap (K:\ell).$$ Since $K\subseteq K:\ell$, we have ${\rm
  ht}(K:\ell)>k-r$. From Proposition \ref{prop:radicalColon}, ${\rm
  ht}(I_{a-1}(\C'))={\rm ht}(I_a(\C):\ell)$, so in order to have ${\rm ht}(I_{a-1}(\C'))=k-r+1$, we must have $(\mathfrak p_i:\ell)=\langle 1\rangle$, for all $i$, which is to say that $\ell\in\mathfrak p$ for all $\mathfrak p\in {\rm Min}_{k-r}(I_a(\C))$. We are unable to analyze this case any further since we do not have any information about $K$ when $r>1$. Note that $d_r(\C)=d_r(\C')$ is equivalent to $p_r'+1=p_r$.

In the rest of this argument, assume that either $j>1$ or $d_r(\C)=d_r(\C')+1$. We have ${\rm ht}(I_a(\C))={\rm ht}(I_{a-1}(\C'))=k-r$. From Proposition \ref{primarydecomp} we have
$$I_a(\C)=\bigcap_{\mathfrak p\in {\rm Min}_{k-r}(I_a(\C))}\mathfrak p^{a-n+\nu(\mathfrak p)}\cap K,$$ where ${\rm ht}(K)>k-r$, and $$I_{a-1}(\C')=\bigcap_{\mathfrak p'\in {\rm Min}_{k-r}(I_{a-1}(\C'))}(\mathfrak p')^{a-n+\nu'(\mathfrak p')}\cap K',$$ where ${\rm ht}(K')>k-r$, and $\nu'(\mathfrak p')$ is the number of linear forms in $\C'$ that belong to $\mathfrak p'$.

Since $I_a(\C)\subset I_{a-1}(\C')$, yet they have the same height equal to $k-r$, one has that ${\rm Min}_{k-r}(I_{a-1}(\C'))\subseteq {\rm Min}_{k-r}(I_a(\C))$. Let $\mathfrak p\in {\rm Min}_{k-r}(I_a(\C))$.

If $\ell\in\mathfrak p$ and $\nu(\mathfrak p)=n-a+1$, then $\mathfrak p\notin {\rm Min}_{k-r}(I_{a-1}(\C'))$. This is because $\nu'(\mathfrak p)=(n-a+1)-1=n-a<(n-1)-(a-1)+1$.

If $\ell\in\mathfrak p$ and $\nu(\mathfrak p)\geq n-a+2$, then $\mathfrak p\in {\rm Min}_{k-r}(I_{a-1}(\C'))$ and $\nu'(\mathfrak p)=\nu(\mathfrak p)-1$.

If $\ell\notin\mathfrak p$, then $\mathfrak p\in{\rm Min}_{k-r}(I_{a-1}(\C'))$ since $n-a+1=(n-1)-(a-1)+1$. Moreover, $\nu(\mathfrak p)=\nu'(\mathfrak p)$, since $\ell\notin \mathfrak p$.

If $\mathfrak q$ is a linear ideal, $m\geq 1$ is any integer, and $\ell$ is a
linear form, then one easily sees that
$$\mathfrak q^m:\ell =\left\{
  \begin{array}{ll}
    \mathfrak q^m, & \mbox{ if } \ell\notin\mathfrak q; \\
    \mathfrak q^{m-1}, & \mbox{ if } \ell\in\mathfrak q.
  \end{array}
\right.$$

With this we have $$I_a(\C):\ell=\left(\bigcap_{\mathfrak p\in A}\mathfrak p^{a-n+\nu(\mathfrak p)-1}\right)\cap\left(\bigcap_{\mathfrak p\in B}\mathfrak p^{a-n+\nu(\mathfrak p)-1}\right)\cap\left(\bigcap_{\mathfrak p\in C}\mathfrak p^{a-n+\nu(\mathfrak p)}\right)\cap (K:\ell),$$ where $A:={\rm Min}_{k-r}(I_a(\C))\setminus {\rm Min}_{k-r}(I_{a-1}(\C'))$, $B:=\{\mathfrak p\in {\rm Min}_{k-r}(I_a(\C))|\ell\in\mathfrak p \mbox{ and }\nu(\mathfrak p)\geq n-a+2\}$, and $C:=\{\mathfrak p\in {\rm Min}_{k-r}(I_a(\C))|\ell\notin\mathfrak p\}$.

From our previous discussion, if $\mathfrak p\in A$, then $\ell\in\mathfrak p$
and $\nu(\mathfrak p)=n-a+1$. So the first term disappears. Also,
$B\cup C={\rm Min}_{k-r}(I_{a-1}(\C'))$, and we have
$$\left(\bigcap_{\mathfrak p\in B}\mathfrak p^{a-n+\nu(\mathfrak
  p)-1}\right)\cap\left(\bigcap_{\mathfrak p\in C}\mathfrak p^{a-n+\nu(\mathfrak
  p)}\right)= \bigcap_{\mathfrak p'\in {\rm Min}_{k-r}(I_{a-1}(\C'))}(\mathfrak
p')^{a-n+\nu'(\mathfrak p')}.$$ So, from Lemma \ref{HilbPoly}(1) we can conclude that in this case
$$\deg(I_a(\C):\ell)=\deg(I_{a-1}(\C')).$$
\end{proof}

\begin{rem} In the proof of Proposition \ref{recdegree}, we could have used our Theorem \ref{main}, coupled with applying Hilbert polynomial to the short exact sequence
\[
 0\longrightarrow \frac{R(-1)}{I_a(\C):\ell}\stackrel{\cdot \ell}\longrightarrow \frac{R}{I_a(\C)}\longrightarrow \frac{R}{\langle x_k,I_a(\mathcal C'')\rangle}\longrightarrow 0,
\]
 and using the standard deletion-contraction formula $$T_{\C}(x,y)=T_{\C'}(x,y)+T_{\C''}(x,y).$$ In order to match the coefficients appropriately, it turns out that the same kind of analysis as we did above must be made, with the same difficulty when studying the case $j=1$ and $p_r'+1=p_r$.
\end{rem}

\vskip .1in

\noindent{\bf Acknowledgement} We are very grateful to the anonymous referee
for a very thorough and constructive report, which led to a significant
improvement in the quality of this manuscript.

\renewcommand{\baselinestretch}{1.0}
\small\normalsize 

\bibliographystyle{amsalpha}

\end{document}